\documentclass{article}
\usepackage[english]{babel}
\usepackage[letterpaper,top=2cm,bottom=2cm,left=3cm,right=3cm,marginparwidth=1.75cm]{geometry}
\usepackage{todonotes} 
\usepackage{amsmath, amssymb, amsthm}
\usepackage{graphicx}
\usepackage{mathtools}
\usepackage[colorlinks=true, allcolors=blue]{hyperref}
\newcommand{\E}{\mathcal{E}}
\newcommand{\R}{\mathbb{R}}
\newcommand{\N}{\mathbb{N}}
\newcommand{\Rd}{\mathbb{R}^d}
\newcommand{\dx}{\mathrm{d}}

\newcommand{\A}{\mathcal{A}}
\newcommand{\B}{\mathcal{B}}

\usepackage{setspace}
\DeclarePairedDelimiter{\norm}{\lVert}{\rVert}
\DeclarePairedDelimiter{\set}{\{}{\}}

\DeclarePairedDelimiter{\ceil}{\lceil}{\rceil}

\DeclarePairedDelimiter{\prt}{(}{)}
\DeclarePairedDelimiter{\brk}{[}{]}

\newtheorem{theorem}{Theorem}[section]
\newtheorem{lemma}[theorem]{Lemma}
\newtheorem{proposition}[theorem]{Proposition}
\newtheorem{corollary}[theorem]{Corollary}

\newtheorem{definition}[theorem]{\bf Definition}

\title{Existence of weak solutions for a volume-filling model of cell invasion into extracellular matrix}
\author{Rebecca M. Crossley\thanks{Mathematical Institute, University of Oxford, Oxford OX2 6GG, UK. (rebecca.crossley@maths.ox.ac.uk).}\and Jan-Frederik Pietschmann${^\dagger},$\thanks{Institute of Mathematics, Universit\"{a}t Augsburg, Institut f\"ur Mathematik, Universit\"{a}tsstra\ss e 12a, 86159 Augsburg, Germany. (jan-f.pietschmann@uni-a.de).}\thanks{Centre for Advanced Analytics and Predictive Sciences (CAAPS), University of Augsburg,
Universit\"atsstr. 12a, 86159 Augsburg, Germany.}\and Markus Schmidtchen\thanks{Institute of Scientific Computing, Technische Universit\"at Dresden, Zellescher Weg 25,
01069 Dresden, Germany. (markus.schmidtchen@tu-dresden.de).}}

\begin{document}
\maketitle

\begin{abstract}
We study the existence of weak solutions for a model of cell invasion into the extracellular matrix (ECM), consisting of a non-linear partial differential equation (PDE) for cell density coupled with an ordinary differential equation (ODE) for ECM density. The model includes cross-species density-dependent diffusion and proliferation terms, capturing the role of the ECM in supporting cells during invasion and preventing growth via volume-filling effects. The occurrence of cross-diffusion terms is a common theme in the system of interacting species with excluded-volume interactions. Additionally, ECM degradation by cells is included. We present an existence result for weak solutions, exploiting the partial gradient flow structure to overcome the non-regularising nature of the ODE. Furthermore, we present simulations that illustrate travelling wave solutions and investigate asymptotic behaviour as the ECM degradation rate tends to infinity.
\end{abstract}

\section{Introduction}
The aim of this work is to analyse a model for cell migration into the extracellular matrix (ECM) that consists of a system including a non-linear, degenerate parabolic, partial differential equation (PDE) for the density of cells, coupled to an ordinary differential equation (ODE) (yet for every point in space) for the ECM.

Mathematical models {of collective cell migration are often implemented to understand how the number of cells within at least one population evolves over} time. 
{However, existing models of this important biological phenomenon become increasingly complex by considering the impacts and influences of various environmental characteristics or additional populations}, such as the availability of nutrients \cite{gatenby1996reaction, mclennan_neural_2015, mclennan_vegf_2015, sherratt_chemotaxis_1994}, surrounding oxygen levels \cite{celora2023spatio}, or physical structures such as the ECM \cite{ crossley2024modeling, dallon_mathematical_1999}. 

The ECM{, in particular,} is an intricate, dynamic network of molecules including elastin, collagen and laminin, {among} others, that varies in structure depending on its location.
The molecules within the ECM are meticulously organised into a complex meshwork that plays an important role in providing structural integrity and repair of tissues, as well as  providing structural and biochemical support to cells during migration \cite{cells13010096}. 
{An} important mechanism of cell-ECM feedback is the degradation of {ECM elements} by surrounding cells through the use of matrix-degrading enzymes, such as membrane-bound and diffusive matrix metalloproteinases \cite{winkler2020concepts}. 
Matrix-degrading enzymes target specific components of the ECM, such as collagen fibers, to degrade them, which, in turn, increases the space available for cells to move into during migration \cite{Brinck2002, Kandhwal2022, Parsons1997}. 

There are a variety of mathematical approaches available for studying the dynamics of cells and the ECM during collective cell migration that are detailed in \cite{crossley2024modeling}. 
However, the specific focus of this work is the following system of differential equations that was derived and non-dimensionalised in \cite{crossley2023travelling} to describe the evolution of a cell density, $u$, and the ECM, $m$, respectively, at a location $x \in \Omega\subset \R^d$, open and bounded, as well as $t\in(0,\infty)$:
\begin{subequations}
\label{eq:main}
\begin{align}
    \frac{\partial u}{\partial t} &=  \nabla \cdot \bigg[ (1- u-m)\nabla u + u \nabla (u+m)\bigg]+ u(1-u-m),\label{eq:main-cells}\\
    \frac{\partial m}{\partial t} &=-\lambda m u.\label{eq:main-ECM}
\end{align}
\end{subequations}
The system \eqref{eq:main} is equipped with initial data assumed to verify
\begin{align*}
    u(x,0) = u_{\mathrm{in}}(x), \qquad \text{ and } \qquad m(x,0) = m_{\mathrm{in}}(x),
\end{align*}
where $u_{\mathrm{in}},\,m_{\mathrm{in}} \in L^2(\Omega)$ satisfy $0 \le u_{\mathrm{in}},m_{\mathrm{in}}$ and $u_{\mathrm{in}} + m_{\mathrm{in}} \le 1$. 
We have chosen to constrain the total mass above by one, in line with the non-dimensional carrying capacity of this model. 
Furthermore, we prescribe no-flux boundary conditions for the cell density, \textit{i.e.},
\begin{align*}
    \bigg[(1- u-m)\nabla u + u\nabla (u+m)\bigg]\cdot {\bf{n}} = 0\qquad \text{ on }\qquad \partial\Omega,
\end{align*}
where ${\bf{n}}$ denotes the outward normal vector on the boundary $\partial \Omega$, as we assume that cells are confined to $\Omega$.
\hfill \vspace{1em}\par
This model follows simplifying assumptions, which are set out below and verified experimentally in \cite{perumpanani1999extracellular}, that state that the timescale over which ECM degradation occurs is far larger than the timescale over which the matrix-degrading enzymes decay, so we can model that cells directly degrade the ECM via their membrane-bound matrix-degrading enzymes only.
As such, Eq.~\eqref{eq:main-ECM} contains only one term that describes the non-dimensionalised degradation rate of ECM by surrounding cells as $\lambda\in\mathbb{R}^{+}$, which largely impacts the speed and structure of travelling wave type solutions observed in this model.

Eq.~\eqref{eq:main-cells} is a reaction-diffusion type PDE that relates closely to the Fisher-Kolmogorov-Petrovsky-Piskunov (Fisher-KPP) equation \cite{fisher_wave_1937, kolmogorov1937study}, which is well-studied in a variety of contexts, including population models in mathematical biology \cite{murray2001mathematical}. 
The first term on the right-hand side of Eq.~\eqref{eq:main-cells} describes the movement of cells down the gradient of the cell density, where movement is prevented by the presence of surrounding cells and ECM. 
The second term models motion of the cells down the gradient of the total density of cells and ECM, ${u}+{m}$, and takes the form of an advective term. 
Diffusion-advection-reaction differential equations with linear diffusion are well-studied in a variety of contexts, see, for example, \cite{perumpanani1995phase, socolofsky2005engineering}.
It is a result of the volume filling assumptions in the underlying individual based model that ensures that the corresponding continuum model includes two flux terms that describe motility{, such as those previously studied in a variety of contexts, which include cross-diffusion \cite{bruna2012diffusion, carrillo2018zoology}, drift \cite{alasio2022existence, alasio2020trend, mason2023macroscopic} and non-local interaction terms \cite{di2018nonlinear}.} 
This is because the volume exclusion processes included prevent cell growth and motility in regions of space with higher density of cells and ECM. 
From a mathematical perspective, we are therefore dealing with a cross-diffusion type of model. 
The final term on the right hand side of Eq.~\eqref{eq:main-cells} captures cell proliferation, which, due to volume-filling effects, is also reduced by the presence of surrounding cells and ECM.

The mathematical analysis of multi-species cross-diffusion models is challenging as, on the one hand, due to the solution-dependent mobilities, such systems are often only degenerate parabolic, as in the current context, which impedes the usual regularising effects. 
On the other hand, the coupling of several equations {often} prevents the use of maximum principles to obtain $L^\infty$-bounds. 
This is particularly important for systems that involve volume exclusion processes, as in this case, since the upper bound on the densities is a central modelling assumption. 
{In the wider context of linear diffusion \cite{kresin2012maximum}, or stationary cross-diffusion \cite{le2023some}, variations of the maximum principle have been employed in order to obtain sufficient regularising bounds \cite{marciniak2010boundedness, tao2014boundedness}. However, they either treat stationary equations or require regularity of solutions which, due to the fact that \eqref{eq:main-cells} is only degenerate parabolic, is not available in our case.}

{However, o}ne, by now established, approach to overcome these difficulties is the so-called boundedness by entropy methods \cite{burger2010nonlinear, JuengelBoundedness2015} that can be applied if the system is (formally) an Otto-Wasserstein gradient flow \cite{JKO1998, Otto1998FirstGradFlow}.
The basic idea is to transform the system into entropy variables, defined as variational derivatives of the entropy functional with respect to the unknown densities. Then, in many practically relevant cases, it turns out that the mapping between these and the original variables automatically enforces the desired bounds. 
This technique has, in different variants, been successfully applied in, for example, \cite{Burger2016, ehrlacher2021existence, zamponi2017analysis}.

Another issue which is specific to the system of equations studied in this work is the fact that Eq.~\eqref{eq:main-ECM}, merely being an ODE, does not contain any diffusive terms and thus, no regularising effect for $m$ can be expected. A similar coupling of PDEs and an ODE was previously analysed in the context of chemotaxis, see, for example, \cite{Corrias2004, MARCINIAK2010, StinnerSurulescuWinkler2013}.
Most importantly in our situation, the coupling also implies that we are not dealing with a formal Otto-Wasserstein gradient flow. A similar situation, yet for a model of tumour growth, was encountered in \cite{BurgerFrielePietschmann2020}.

As a consequence, the main difficulty in studying the aforementioned system stems from the fact that it cannot be cast into {any standard} gradient flow framework. Indeed, we first observe that the form of the flux inside the divergence in Eq.~\eqref{eq:main-cells} is exactly the same as for the volume filling model analysed in \cite{burger2010nonlinear}. 
The entropy functional that the authors used in their work, yet without self-diffusion for $m$, reads as 
\begin{equation}
    \label{eq:entropy-cts}
    \E (u,m) = \int_\Omega u(\log(u)-1) +(1-\rho)(\log(1-\rho) - 1) \dx x,
\end{equation}
with $\rho(x,t) := u(x,t) + m(x,t)$. 
This allows us to rewrite Eq.~\eqref{eq:main-cells} in the form
\begin{equation*}
    \frac{\partial u}{\partial t} = \nabla \cdot \bigg[u(1-\rho) \nabla \frac{\delta \E}{\delta u}\bigg] + u(1-\rho). \label{Eeqn_u}
\end{equation*}
{
The form of the right-hand side of Eq. (1a) dictates an Otto-Wasserstein-type gradient flow, which, in turn, implies the functional form of the free energy with respect to $u$, up to a constant, possibly depending on $m$. However, the reaction term in Eq.~\eqref{eq:main-cells} disrupts this gradient flow structure as the equation is not mass-conserving. 
Coupled to Eq.~\eqref{eq:main-ECM} it is not clear in which sense the whole system could be interpreted as gradient flow. 
Since the system does not fall into a remotely standard gradient flow framework that is often exploited in existence proofs.
} 
%Thus, the system as a whole is not a gradient flow due to the coupling of the cell density equation with the equation for the ECM. 
In particular, we do not expect the entropy (Eq.~\eqref{eq:entropy-cts}) to act as a Lyapunov functional. 
Nevertheless, we can show that it grows at most linearly in time. 
Then, we are able to extract regularity information from its dissipation which provides adequate compactness to establish the existence of solutions to the system. 
These estimates are made rigorous by means of a time-discretistion and regularisation procedure.

The cornerstone of this analysis is the observation that there is a regularising effect for the ECM equation despite it only being an ODE. 
Albeit not being a gradient flow, from the dissipation we obtain an estimate of the form
$$
    \|\nabla u\|^2_{L^2({t_{n-1},t_n}; L^2(\Omega))} - \|\nabla m\|^2_{L^2({t_{n-1},t_n}; L^2(\Omega))} \leq C,
$$
in conjunction with the conditional regularising effect for the ECM density, \textit{i.e.},{
\begin{align*}
    \norm{\nabla m}_{L^\infty({t_{n-1},t_n};L^2(\Omega))}^2 
    \leq 2^n \norm{\nabla m_0}_{L^2(\Omega)}^2 + 2 \tau \lambda^2 \Delta t \sum_{i=0}^{n-1} 2^{n-i} \norm{\nabla u}^2_{L^2({t_i,t_{i+1}}; L^2(\Omega))},
\end{align*}
}
on any time interval $[t_n, t_{n+1}]$ of length $\Delta t>0${, see Lemma \ref{lem:conditional-grad-m}.} 
It requires a fine study to combine both estimates in a suitable way that allows extension to the whole time interval $[0,T]$. 

\subsection{Statement of the main theorem and structure of the paper}
Before presenting the main result of this work, let us introduce our notion of a weak solution.
\begin{definition}[Weak solution]
	\label{def:weak-soln}
	For given initial data $(u_0,m_0) \in L^2(\Omega) \times H^1(\Omega)$ with $0 \leq m_0, \,u_0$ such that $u_0 + m_0 \leq 1$ a.e. in $\Omega$, we call $(u,m)$ a \emph{weak solution} to Eqs.~\eqref{eq:main}, if $ u,m\in L^2(0,T;H^1(\Omega))$ such that $\partial_t u \in L^2(0,T;(H^1(\Omega))')$ and $\partial_t m \in L^2(0,T;L^2(\Omega))$, and if, for any $\varphi \in H^1(\Omega)$, there holds
	\begin{align*}
		&\int_0^T\int_\Omega \partial_t u \, \varphi \, \dx x \dx t+ \int_0^T \int_\Omega [(1-u-m)\nabla u + u \nabla (u+m)] \cdot \nabla \varphi \, \dx x \dx t  \nonumber \\
  &\qquad\qquad\qquad\qquad\qquad\qquad\qquad= \int_0^T \int_\Omega u (1-u-m)\varphi \, \dx x \dx t,
	\end{align*}
	and $u(\cdot,0)=u_0$ as well as
	\begin{align*}
		\frac{\partial m}{\partial t} = - \lambda m u,
	\end{align*}
	for almost every $x \in \Omega$ and $t\in [0,T]$, as well as $m(\cdot, 0) = m_0$.
\end{definition}

Pursuing the strategy outlined above, we are able to prove the following existence result.
\begin{theorem}[Existence]
\label{thm:main} Given $(u_0,m_0) \in L^2(\Omega) \times H^1(\Omega)$ such that $0 \le m_0, \,u_0$ and $u_0 + m_0 \le 1$ a.e. in $\Omega$, there exists a weak solution $(u,m) \in (L^2(0,T;H^1(\Omega)))^2$ to system~\eqref{eq:main} in the sense of Definition~\ref{def:weak-soln}, which also satisfies the box constraints 
$$
0 \le u,\,m \text{ and } u + m \le 1\text{ a.e. in }\Omega \times (0,T).
$$
In addition, there exist non-negative constants {$C_\mathrm{u}$ and $C_\mathrm{m}$, depending on $\Omega$, $T$ and initial data $(u_0,m_0) \in L^2(\Omega) \times H^1(\Omega)$, only, s.t. 
\begin{align}\label{eq:a_priori}
    \|\nabla u\|_{L^2(0,T;L^2(\Omega))}^2  \leq C_\mathrm{u},\quad \text{  and  }\quad \|\nabla m\|_{L^2(0,T;L^2(\Omega))}^2  \leq C_\mathrm{m}.
\end{align}
}
\end{theorem}
The rest of the paper is organised as follows. 
Section~\ref{sec:proof} is devoted to the proof of Theorem~\ref{thm:main} and subdivided into the construction of the time-discrete, regularised approximation (Section~\ref{sec:time_discrete}), the rigorous derivation of a-priori estimates (Section~\ref{sec:a_prioi}) and finally, the passage to the limit to recover a weak solution (Section~\ref{sec:exist-weak-sol}). 

In Section~\ref{sec:numerics}, we then provide numerical examples of the travelling wave solutions that are exhibited in the system~\eqref{eq:main} and investigate the limiting behaviour as $\lambda \to\infty$.

\section{Proof of the main theorem}
\label{sec:proof}
The first step of our proof consists of constructing a time discrete, regularised (both in the equations and the initial datum for the ECM equation) approximation to the system \eqref{eq:main}, which is, however, still non-linear. Therefore, we have to show the existence of iterates for this scheme by means of a fixed point argument. It turns out that a regularisation of the initial datum $m_{\mathrm{in}}$ results in solutions satisfying the box constraints strictly. This allows us to obtain a time-discrete, yet rigorous, version of the entropy--entropy-dissipation inequality. Thus, using a careful construction to mitigate the lack of regularity in $m$, we eventually obtain bounds that are sufficient to pass to the limit and obtain the desired weak solution. Throughout our analysis we will use the letter $C$ to denote constants that may change from line to line.

\subsection{Time discretisation and regularisation}
\label{sec:time_discrete}
Let $N\in\mathbb{N}$, and discretise time such that $(0,T]$ has sub-intervals of the form 
\begin{equation*}
    (0,T]=\bigcup_{k=1}^N \big((k-1)\tau, k\tau\big],
\end{equation*}
where $\tau=\frac{T}{N}$. 
We write the recursive sequence 
\begin{subequations}
\label{eq:implicit-euler}
\begin{align}
    \frac{u_k-u_{k-1}}{\tau}&=\tau {(\Delta w_{k} - w_{k})}+\nabla \cdot \big(u_{k}(1-\rho_{k})\nabla w_{k}\big)+u_{k}(1-\rho_{k}), \label{eq:discretised_u} \\
    \frac{m_k-m_{k-1}}{\tau} &= - \lambda m_k u_k,\label{eq:discretised_m}
\end{align}
\end{subequations}
where $\rho_k := u_k + m_k$, $w_k = w(u_k, m_k)$ and with regularised initial data 
$$
		m_{0} :=  \max(\tau, \min(m_{\mathrm{in}}, 1 - \tau)),\quad \text{and} \quad u_0 = u_{\mathrm{in}}.
$$
Furthermore, we recall that
\begin{equation}
\begin{aligned}
    \label{eq:defn-entr-var-w}
    w(u,m)&\coloneqq \frac{\delta \E}{\delta u}= \log(u)-\log(1-\rho), 
\end{aligned}
\end{equation}
such that the transformation from $(u,m)$ to the entropy variable $w$ is given by
\begin{equation*}
    u=\frac{(1-m)e^{w}}{1+e^{w}}= \frac{(1-m)}{1+e^{-w}}. \label{eq:transform_w_to_u}
\end{equation*}
We note that Eq.~\eqref{eq:discretised_m} can be solved for $m_k$ when given $m_{k-1}$ and $u_k$ as 
\begin{equation}
 m_k = \frac{m_{k-1}}{1+\lambda \tau u_k}. \label{eq:discretised_m_sol}   
\end{equation}
In order to show the existence of weak solutions, we define the set 
$${
    \A:= \left\{ (u,m) \in (L^\infty(\Omega))^2: \, 0 \le u, \, m \le 1, \, 0 \le \rho = u + m  \le 1 \right\}.}
$$
\begin{lemma}\label{lem:s1}
Given $\tau > 0$ and  $(\tilde u, \tilde m) \in \A$, we set $\tilde\rho = \tilde u + \tilde m$. Then, the linear problem 
\begin{align}
    \label{eq:discrete_linearized}
    \int_\Omega (\tau + \tilde{u}(1-\tilde\rho))\nabla w\cdot \nabla \varphi + \tau w\varphi \; \dx x = \int_\Omega \left[\tilde u(1-\tilde \rho) - \frac{\tilde u - u_{k-1}}{\tau} \right]\varphi\; \dx x,
\end{align}
for all $\varphi \in H^1(\Omega)$, has a unique solution $w \in H^1(\Omega)$ such that
\begin{align}
    \label{eq:discrete_linearized_a_priori}
    \|w\|_{H^1(\Omega)} \le C,
\end{align}
where the constant $C>0$ depends only on $\tau$ and $u_{k-1}$.

In addition, the operator $S_1 : \A \to L^2(\Omega)$, which assigns some given $(\tilde u, \tilde m) \in \A$ to $w$, being the solution to Eq.~\eqref{eq:discrete_linearized} is continuous and compact. 
\end{lemma}
\begin{proof}
As $\tilde u(1-\tilde \rho) \ge 0$ and $\tau > 0$, the existence  of a unique solution $w\in H^1(\Omega)$ is a direct consequence of the Lax-Milgram Lemma, cf. \cite{Brezis2010}. The a priori bound follows by choosing $\varphi = w$ as a test function and applying a weighted Young inequality to the right-hand side of Eq.~\eqref{eq:discrete_linearized}.

For the continuity of $S_1$, consider a sequence $(\tilde u_n, \tilde m_n) \in \A$ such that $(\tilde u_n, \tilde m_m) \to (\tilde u, \tilde m) \in \A$ and denote by $w_n$ and $w$ the respective solutions to Eq.~\eqref{eq:discrete_linearized}. Subtracting the respective equations and choosing $\varphi = (w-w_n)$ yields
\begin{align}
    \label{eq:discretized_continuity}
    &\int_\Omega (\tau + \tilde u(1-\tilde \rho)|\nabla(w-w_n)|^2 + \tau |w-w_n|^2\dx x\nonumber\\
    &\qquad = \int_\Omega [\tilde u (1-\tilde \rho) - \tilde u_n(1-\tilde \rho_n)]\nabla w_n\cdot \nabla (w-w_n) - \frac{\tilde u - \tilde u_n}{\tau}(w-w_n) \dx x.
\end{align}
As both $(\tilde u_n, \tilde m_n) \in \A$ and $(\tilde u, \tilde m) \in \A$, elementary calculations show that the following estimate holds:
$$
|\tilde u(1-\tilde \rho)-\tilde u_n(1-\tilde \rho_n)| \le 2|\tilde u - \tilde u_n| + |\tilde m - \tilde m_n|. 
$$
Additionally, we can use that  $0 \le \tilde u(1-\tilde \rho) \le 1$ and estimate the left-hand side in Eq.~\eqref{eq:discretized_continuity} from below to obtain
\begin{align*}
    &\tau \|w-w_n\|_{H^1(\Omega)}^2 \leq \int_\Omega [\tilde u (1-\tilde \rho) - \tilde u_n(1-\tilde \rho_n)]\nabla w_n\cdot \nabla (w-w_n) - \frac{\tilde u - \tilde u_n}{\tau}(w-w_n) \dx x \\
    &\quad\leq \prt*{2\|\tilde u - \tilde u_n\|_{L^\infty(\Omega)} + \|\tilde m - \tilde m_n\|_{L^\infty(\Omega)}}\|\nabla w_n\|_{L^2(\Omega)}\|\nabla (w-w_n)\|_{L^2(\Omega)} \\
    &\qquad\qquad\qquad\qquad+ \frac{\sqrt{|\Omega|}}{\tau}\|\tilde u - \tilde u_n\|_{L^\infty(\Omega)}\|w-w_n\|_{L^2(\Omega)}\\
    &\quad\le C' \prt*{\|\tilde u - \tilde u_n\|_{L^\infty(\Omega)} + \|\tilde m - \tilde m_n\|_{L^\infty(\Omega)}}\|w-w_n\|_{H^1(\Omega)},
\end{align*}
where $C'>0$ only depends on $\tau$, $\Omega$ and $C>0$ comes from Eq.~\eqref{eq:discrete_linearized_a_priori}. 
Thus, as $(\tilde u_n, \tilde m_n) \to (\tilde u,\tilde m)$ in $L^\infty(\Omega)$, we conclude that $w_n \to w$ in $H^1(\Omega)$.
Finally, the compactness of $S_1$ then follows from the compactness of the embedding $H^1(\Omega) \hookrightarrow L^2(\Omega)$.
\end{proof}

Our next goal is to recover a new pair $(u,m)$ from $w$. To this end, we observe that, by definition of $w$, we have 
\begin{align}
    \label{eq:u_from_m}
    u = (1-m)\frac{e^{w}}{1+e^{w}} =: a_{w}(1-m).
\end{align}
Inserting this relation into Eq.~\eqref{eq:discretised_m} yields 
\begin{align}
    \label{eq:discretized_m_nonlinear}
    \frac{m-m_{k-1}}{\tau} = -\lambda um = -\lambda a_{w}m(1-m)
\end{align}
The next result shows that, indeed, this equation is uniquely solvable.
\begin{lemma}
\label{lem:s2}
Given $m_{k-1}$ that satisfies $0 \le  m_{k-1} \le 1$, and for $\tau$ sufficiently small but independent of $m_{k-1}$, Eq.~\eqref{eq:discretized_m_nonlinear} has a unique solution $m \in [0,1]$.

Additionally, the operator $S_2 : L^2(\Omega) \to \A$, which maps $w$ to $(u,m)$ where $u$ is calculated from $m$ and $w$ via Eq.~\eqref{eq:u_from_m}, is well-defined and continuous.
\end{lemma}
\begin{proof}
To establish the existence of a unique solution $m\in [0,1]$ to Eq.~\eqref{eq:discretized_m_nonlinear}, we define the set 
$$
\B = \{ z \in L^\infty(\Omega) \; : \; 0 \le z \le 1\},
$$
and the mapping 
$$
    K:\B \to \B, \quad K(z) = \frac{m_{k-1}}{1+\tau\lambda a_w (1-z)}.
$$
As $0 \le  m_{k-1} \le 1$ and $(1-z) \ge 0$ almost everywhere, $K$ is well-defined,  continuous and every fixed point of $K$ is a solution to Eq.~\eqref{eq:discretized_m_nonlinear}.
Furthermore, we can estimate
\begin{align*}
    \|K(z_1) - K(z_2)\|_{L^\infty(\Omega)} \leq  \tau \lambda a_w \|m_{k-1}\|_{L^\infty(\Omega)}\|z_1-z_2\|_{L^\infty(\Omega)},
\end{align*}
thus, as $a_w \le 1$, for $\tau$ sufficiently small, yet independent of $m_{k-1}$, $K$ is a contraction. Applying Banach's fixed-point theorem we conclude the existence of a unique fixed point $m \in [0,1]$ which is then a solution to Eq.~\eqref{eq:discretized_m_nonlinear}. Finally, the continuity of $S_2$ follows from standard computations.
\end{proof}
We are now in a position to prove the existence of iterates for Eqs.~\eqref{eq:discretised_u} and \eqref{eq:discretised_m}.
\begin{theorem}
	\label{thm:fixed-point}
For given $(u_{k-1},m_{k-1}) \in \A$ with $m_{k-1} \in H^1(\Omega)$ and $\tau$ sufficiently small, there exists {a fixed point $(u_k,m_k) \in \A $. Moreover, $(u_k,m_k) \in (H^1(\Omega))^2$}.
\end{theorem}
\begin{proof}
    Due to the results of Lemma~\ref{lem:s1} and Lemma~\ref{lem:s2}, the operator 
    $$
    S = S_2 \circ S_1 : \A \to \A
    $$
    is well-defined, continuous and compact. 
    Furthermore, it is readily observed that $\A$ is a convex subset of $(L^\infty(\Omega))^2$. 
    Thus, an application of Schauder's fixed point theorem yields the existence of a fixed point $(u_k,m_k) \in \A$ associated to $w_k = S_1(u_k, m_k)$. 
    
    {To show that the fixed-point has $H^1$-regularity, we obtain from Eq.~\eqref{eq:discretized_m_nonlinear} that}
    $$
        \nabla m_k \prt*{\frac1\tau + \alpha_{w_k} \lambda (1-2m_k)} = \frac{\lambda m_k (1-m_k)}{(1+e^{w_k})^2} \nabla w_k + \frac1\tau \nabla m_{k-1}.
    $$
    For sufficiently small $\tau > 0$, the parenthesis on the left-hand side is positive, and we find
    \begin{align}
	    \label{eq:regularity-m}
        \norm{\nabla m_k}_{L^2(\Omega)} 
        \leq C (\norm{\nabla w_k}_{L^2(\Omega)} + \norm{\nabla m_{k-1}}_{L^2(\Omega)}),
    \end{align}
    where $C>0$ only depends on $\lambda, \tau>0$ and therefore $m_k \in H^1(\Omega)$. 
    Furthermore, passing to the gradient in Eq.~\eqref{eq:u_from_m}, we also obtain the desired $H^1$-bound for $u_k${, which concludes the proof.}
\end{proof}

\begin{theorem}[Uniform upper and lower bounds on $u_k,\,m_k$]\label{thm:strict_bounds}
	\label{thm:stability}
	Let $(u_k, m_k)$ be solutions provided by Theorem \ref{thm:fixed-point} and $\tau >0$. Then there holds
 $$
    0 < u_k, \, m_k < 1, \quad \text{ as well as }\quad  0 < \rho_k < 1,
 $$
 for all {$k \in \N^{>0}$}.

\end{theorem}
The strict bounds {in Theorem~\ref{thm:strict_bounds}} are necessary in order to rigorously derive the entropy-dissipation inequality later on.
\begin{proof}
Starting from $(u_{0}, m_{0})$, we construct $w_1$ which solves
	\begin{align}
		\label{eq:w-k-ell1}
	    \int_\Omega (\tau + u_1(1-\rho_1))\nabla w_1 \cdot \nabla \varphi + \tau w_1 \varphi \;\dx x = \int_\Omega \left[ u_1 (1 - \rho_1) - \frac{ u_1 - u_{0}}{\tau} \right]\varphi\;\dx x.
	\end{align}
	{As the right-hand side is bounded in $L^\infty(\Omega)$, say, by the constant $\tilde C_\tau>0$, depending on $\tau$, we have that
	\begin{align*}
		\norm{w_1}_{L^\infty(\Omega)} \leq C_\tau := \frac{\tilde C_\tau}{\tau}
	\end{align*}
	Thes upper bound  bounds follow by using $(w_1 - C_\tau)_+$ as test functions in Eq.~\eqref{eq:w-k-ell1}.  Indeed, in the first case we end up with 
    \begin{align*}
        \tau \int_\Omega (w_1 - C_\tau)_+^2\dx x &= -\int_\Omega (\tau + u_1(1-\rho_1))\left|\nabla (w_1 - C_\tau)_+\right|^2\;\dx x\\
        &\quad + \int_\Omega \left[ u_1 (1 - \rho_1) - \frac{ u_1 - u_{0}}{\tau} \right](w_1 - C_\tau)_+\;\dx x - \tau\int_\Omega C_\tau(w_1 - C_\tau )_+\dx x\\
        & \le \int_\Omega (\tilde C_\tau - \tau C_\tau)(w_1 - C_\tau)_+\dx x= 0,
    \end{align*}
    as $\tilde C_\tau = \tau C_\tau$. The lower bound follows analogously.
    }
    Note that then
	\begin{align}\label{eq:bounds_on_a_w}
		a_{w_1} = \frac{e^{w_1}}{1 + e^{w_1}} \in \brk*{e^{-C_\tau}, \frac{1}{1+e^{-C_\tau}}}.
	\end{align}
	In addition, note that $m_1$ is related to $w_1$ via Eq.~\eqref{eq:discretized_m_nonlinear}
	\begin{align*}
    	\frac{m_1-m_0}{\tau} = -\lambda a_{w_1}m_1 (1 - m_1).
	\end{align*}	
	Thus, on the one hand, we have
	$$
		\frac{m_1-m_0}{\tau} \leq 0,
	$$
	i.e. $m_1 \leq m_{0} \leq 1 - \tau$. Conversely, we have, as $a_w \le 1$, that
	\begin{align*}
    	\frac{m_1-m_0}{\tau} 
    	&= -\lambda a_{w_1} m_1 (1 - m_1)\geq  -\lambda m_1,
	\end{align*}
	which implies
	$$
		m_1 \geq \frac{\tau}{(1+\tau \lambda)} >0.
	$$
	Recalling Eq.~\eqref{eq:u_from_m} then yields
	\begin{align*}
		u_1 = a_{w_1}(1-m_1)\in (0,1),
	\end{align*}
	owing to the fact that $m_1 \in (0,1)$ and $a_{w_1}\in(0,1)$. For 
$$
    \rho_1 = u_1 + m_1 = a_{w_1} + m_1(1-a_{w_1})
$$
this implies the bounds
\begin{align*}
    0 < e^{-C_\tau} + \frac{\tau}{(1+\tau \lambda)} \le  a_{w_1} + m_1(1 - a_{w_1}) = \rho_1
\end{align*}
as well as 
\begin{align*}
    \rho_1 &= a_{w_1} + m_1 (1 - a_{w_1}) \leq a_{w_1} + (1 - \tau) (1 - a_{w_1}) \\
    &= 1 - \tau (1 - a_{w_1}) \leq 1 - \tau\left(1- \frac{1}{1 + e^{-C_\tau}}\right) < 1.
\end{align*}
Iterating this procedure, we obtain 
\begin{align*}
    \frac{\tau^k}{(1+\tau \lambda)^k} \le m_k \le 1-\tau, 
\end{align*}
as well as
\begin{align*}
    e^{-C_\tau} \le u_k \le 1 -\frac{\tau^k}{(1+\tau \lambda)^k},
\end{align*}
and
\begin{align*}
    e^{-C_\tau} + \frac{\tau^k}{(1+\tau \lambda)^k} \le  \rho_k \le 1 - \tau\left(1-\frac{1}{1+e^{-C_\tau}}\right),
\end{align*}
where  $C_\tau \to \infty$ as $\tau \to 0$.
\end{proof}

{
\begin{corollary}
    The fixed point $(u_k,m_k) \in \A \cap (H^1(\Omega))^2$ constructed above is a weak solution to Eqs.~\eqref{eq:discretised_u}~and~\eqref{eq:discretised_m}.
\end{corollary}
\begin{proof}
    In view of the bounds provided by Theorem \ref{thm:strict_bounds}, we can invert Eq.~\eqref{eq:u_from_m} and obtain $w_k = \log(u_k) - \log(1- \rho_k)$, which allows the identification of $(u_k, m_k)$ as a weak solution of Eqs.~\eqref{eq:discretised_u}~and~\eqref{eq:discretised_m}.
\end{proof}
}
\subsection{A priori estimates}
\label{sec:a_prioi}
Let $(u_k, m_k)_{k=0}^\infty \subset \A \, \cap \, (H^1(\Omega))^2$ be the sequence obtained from the implicit Euler approximation, cf. Eq.~\eqref{eq:implicit-euler}. 
This section is dedicated to deriving estimates independent of the time step size, $\tau>0$, that ultimately provide sufficient compactness to obtain weak solutions to the continuous system, Eq.~\eqref{eq:main}. 

\begin{proposition}
    \label{prop:discrete-entropy}
    Let $(u_k, m_k)_{k=0}^\infty \subset \A \cap (H^1(\Omega))^2$ as defined in Theorem \ref{thm:stability}.  Then, for any $k \in \mathbb N$, the following discrete entropy estimate holds
    \begin{align}
        \label{eq:entropy-dissipation-discrete}
        &\frac1\tau\prt*{\mathcal E(u_k, m_k) - \mathcal E(u_{k-1}, m_{k-1}) }
        + \tau \int_\Omega |\nabla w_k|^2 + |w_k|^2 \, \dx x  \nonumber \\ & \qquad\qquad\qquad\qquad\leq  C - \int_\Omega u_{k}(1-u_k -m_k) |\nabla w_k|^2 \, \dx x,
    \end{align} 
    where $C>0$ only depends on $\Omega$. Moreover,
    \begin{align}
        \label{eq:grad-u-by-grad-m}
        \|\nabla u_k\|_{L^2(\Omega)}^2 \leq C +   \frac1\tau \prt*{\mathcal E(u_{k-1}, m_{k-1}) - \mathcal E(u_k, m_k)} + \|\nabla m_k\|_{L^2(\Omega)}^2.
    \end{align}
    with $C=C(\tau)>0$ independent of $k\in \N$.
\end{proposition}
\begin{proof}
Owing to the strict upper and lower bounds provided by Theorem~\ref{thm:strict_bounds}, the logarithmic terms appearing in the derivative of the entropy are well defined. Thus, we can use the joint convexity of the energy to obtain
    \begin{align*}
        \mathcal E(u_k, m_k) - \mathcal E(u_{k-1}, m_{k-1}) &\leq \int_\Omega \brk*{\log u_k - \log(1 - u_k - m_k)}(u_{k} - u_{k-1}) \dx x \\
        &\quad - \int_\Omega \log(1 - u_k - m_k) (m_{k} - m_{k-1}) \dx x\\
        &\leq \int_\Omega \brk*{\log u_k - \log(1 - u_k - m_k)}(u_{k} - u_{k-1}) \dx x,
    \end{align*}
    where the last term was discarded due to its sign. Using the definition of the entropy variable, Eq.~\eqref{eq:defn-entr-var-w}, this results in
    \begin{align*}
        \mathcal E(u_k, m_k) - \mathcal E(u_{k-1}, m_{k-1}) \leq \int_\Omega w_k (u_k - u_{k-1} )  \dx x.
    \end{align*}
	Due to the $H^1(\Omega)$-regularity of $w_k$, we may use it as a test functions in Eq.~\eqref{eq:implicit-euler} to get 
    \begin{align*}
        \frac1\tau\prt*{\mathcal E(u_k, m_k) - \mathcal E(u_{k-1}, m_{k-1}) }
        &\leq  - \tau \int_\Omega \prt*{|\nabla w_k|^2 + |w_k|^2} \dx x\\
        &\quad - \int_\Omega \nabla w_k \cdot (u_{k}(1-u_k -m_k) \nabla w_k) \dx x \\
        &\quad + \int_\Omega w_k u_k (1 - u_k - m_k) \dx x.
    \end{align*}
    Integrating by parts, and using that $|\Omega|<\infty$ as well as $0\leq u_k,m_k \leq 1$ and the definition of $w_k$, the inequality further simplifies to 
    \begin{align*}
        &\frac1\tau\prt*{\mathcal E(u_k, m_k) - \mathcal E(u_{k-1}, m_{k-1}) } \nonumber \\
        &\qquad\leq  C - \tau \int_\Omega |\nabla w_k|^2 + |w_k|^2 \dx x - \int_\Omega u_{k}(1-u_k -m_k) |\nabla w_k|^2 \dx x,
    \end{align*}
    where $C>0$ is independent of $k, \ell, \tau$.
    Using the definition of $w_k$, cf. Eq \eqref{eq:defn-entr-var-w}, we further estimate 
    \begin{align*}
        &\int_\Omega u_{k}(1 - u_k - m_k) |\nabla w_k|^2 \dx \nonumber \\  &\qquad\qquad= \int_\Omega (1 - \rho_k)\frac{|\nabla u_k|^2}{u_k} + 2\nabla u_k\cdot \nabla( u_k + m_k) + u_k\frac{|\nabla(1 - \rho_k)|^2}{1 - \rho_k}\dx x\\
        &\qquad\qquad\ge \int_\Omega 2|\nabla u_k|^2 - 2|\nabla u_k| |\nabla m_k|\;\dx x \ge \int_\Omega |\nabla u_k|^2 - |\nabla m_k|^2 \dx x.
    \end{align*}
    Inserting this estimate into the entropy inequality above yields 
    \begin{align*}
		\|\nabla u_k\|_{L^2(\Omega)}^2 \leq C +   \frac1\tau \prt*{\mathcal E(u_{k-1}, m_{k-1}) - \mathcal E(u_k, m_k)} + \|\nabla m_k\|_{L^2(\Omega)}^2.
    \end{align*}
\end{proof}

Before proceeding, let us introduce the piecewise constant interpolations associated to the iterates of the implicit Euler method. 
For $v \in \{u, m, w\}$, we define
\begin{align*}
    v_\tau(x,t) := v_k(x),
\end{align*}
whenever $t\in(t_{k-1}, t_{k}]$, for all $k \geq 1$ and $x \in \Omega$.

\begin{corollary}[$H^1$-estimate for $w_\tau$]
	\label{cor:L2H1-estimate-w}
    Let $(u_k, m_k)_{k=0}^\infty \subset \A \cap (H^1(\Omega))^2$ and $(w_k)_{k=0}^\infty$ be the associated entropy variables defined in Eq.~\eqref{eq:defn-entr-var-w}. Then
    \begin{align*}
        \tau^{1/2} \|w_\tau\|_{L^2(0,T; H^1(\Omega))} \leq C,
    \end{align*}
    where $w_\tau$ denotes the corresponding piecewise constant interpolation and the constant $C>0$ only depends on $\Omega$ and $T>0$.
\end{corollary}

\begin{proof}
    From Eq.~\eqref{eq:entropy-dissipation-discrete} we get
    \begin{align*}
        \tau \prt*{\|\nabla w_k\|_{L^2(\Omega)}^2 + \|w_k\|_{L(\Omega)}^2}  \leq C +   \frac1\tau \prt*{\mathcal E(u_{k-1}, m_{k-1}) - \mathcal E(u_k, m_k)},
    \end{align*}
    which, upon integrating over $(t_{k-1}, t_k]$ and summing from $k=1, \ldots, N_T$, gives
    \begin{align*}
        &\tau \prt*{\|\nabla w_\tau\|_{L^2(0,T;L^2(\Omega))}^2 + \|w_\tau\|_{L^2(0,T;L^2(\Omega))}^2} \nonumber \\ &\qquad\qquad\leq CT + \tau \sum_{k=1}^{N_T} \frac1\tau\prt*{\mathcal E(u_{k-1}, m_{k-1}) - \mathcal E(u_k, m_k)}.
    \end{align*}
    Since the sum on the right-hand side is telescopic, the estimate simplifies further to
    \begin{align*}
        \tau \prt*{\|\nabla w_\tau\|_{L^2(0,T;L^2(\Omega))}^2 + \|w_\tau\|_{L^2(0,T;L^2(\Omega))}^2} \leq CT +  \mathcal E(u_{0}, m_0) - \mathcal E(u_{N_T}, m_{N_T}).
    \end{align*}
    Since $0\leq u_k, m_k \leq 1$ and $|\Omega| \leq C$, we obtain the following uniform estimate
    \begin{align*}
        \tau^{1/2} \|w_\tau\|_{L^2(0,T; H^1(\Omega))} \leq C,
    \end{align*}
    independent of $\tau>0$.
\end{proof}

\begin{lemma}[Conditional estimate for $\nabla m$]
    \label{lem:conditional-grad-m}
    Let $(u_k, m_k)_{k=0}^\infty \subset \A \cap (H^1(\Omega))^2$ be the solution to the implicit Euler approximation, Eq.~\eqref{eq:implicit-euler}. 
    Then, for any $n, \ell \in \N$, we have the conditional estimate
   \begin{align*}
    \norm{\nabla m_{n\ell}}_{L^2(\Omega)}^2 &\leq 2^{n} \norm{\nabla m_0}_{L^2(\Omega)}^2+\tau^2 \lambda^2 \ell  \sum_{{k=0}}^{n-1} 2^{n-k}\phi_k. \label{eq:normm_nl}
    \end{align*} 
    where 
    \begin{equation}
        \label{eq:phi}
        \phi_k = \sum_{i=k\ell+1 }^{(k+1)\ell} \tau \norm{\nabla u_{i}}_{L^2(\Omega)}^2.
    \end{equation}
\end{lemma}
\begin{proof}
Let $j \in \mathbb N$ be fixed. Recalling  Eq.~\eqref{eq:discretised_m_sol}, we can write, for $r=1,\, \dots, \, \ell, $ that
\begin{align*}
    m_{j+r} 
    % &= m_{j+r-1}\dfrac{1}{1+\tau \lambda u_{j+r}}, \\
    % &= m_{j+r-2}\dfrac{1}{1+\tau \lambda u_{j+r-1}}\dfrac{1}{1+\tau \lambda u_{j+r}}, \\
    % &\ \  \vdots\\
    % &=m_{j} \prod_{i=1}^{r} \dfrac{1}{1+\tau \lambda u_{j+r+1-i}}, \\
    &=m_{j} \prod_{i=1}^{r} \dfrac{1}{1+\tau \lambda u_{j+i}}.
\end{align*}
Taking the gradient of this expression, we find that 
\begin{align*}
    \nabla m_{j+r}&=\nabla m_j \prod_{i=1}^{r}\dfrac{1}{1+\tau \lambda u_{j+i}} - m_{j}\sum_{i=1}^{r}\prt*{\prod_{j\neq i}\dfrac{1}{1+\tau \lambda u_{j+i}}}\dfrac{\tau\lambda \nabla u_{j+i}}{(1+\tau\lambda u_{j+i})^2}. 
\end{align*}
Finally, we pass to the norm and obtain
\begin{align}
    \norm{\nabla m_{j+r}}_{L^2(\Omega)} 
    \leq \norm{\nabla m_j}_{L^2(\Omega)}+\tau \lambda \sum_{i=1}^{r}\norm{\nabla u_{j+i}}_{L^2(\Omega)}, \label{eq:grad_mjr}
\end{align}
having used the fact that $0 \leq u_k, m_k \leq 1$. 
Squaring both sides yields
\begin{align}
    \label{eq:sup-grad-m}
    \sup_{1\leq r\leq\ell-1} \norm{\nabla m_{j+r}}_{L^2(\Omega)}^2 &\leq 2\norm{\nabla m_j}_{L^2(\Omega)}^2+2\lambda^2\tau^2 \ell \sum_{i=1}^{\ell}\norm{\nabla u_{j+i}}_{L^2(\Omega)}^2.
\end{align}
Next, we set $j = n\ell$, upon which the preceding estimate becomes
\begin{align}
    \sup_{n \ell \leq k\leq(n+1)\ell} \norm{\nabla m_{k}}_{L^2(\Omega)}^2 &\leq 2\norm{\nabla m_{n\ell}}_{L^2(\Omega)}^2+2\lambda^2\tau^2 \ell\phi_n , \label{eq:sup_discretised_m3}
\end{align}
where 
\begin{equation*}
    \phi_n = \sum_{k=n\ell+1 }^{(n+1)\ell} \tau \norm{\nabla u_{k}}_{L^2(\Omega)}^2.
\end{equation*}
If we now consider Eq.~\eqref{eq:sup_discretised_m3}, we can write
\begin{align*}
    \norm{\nabla m_{n\ell}}_{L^2(\Omega)}^2 &\leq \sup_{(n-1) \ell \leq k\leq n\ell} \norm{\nabla m_{k}}_{L^2(\Omega)}^2
    \leq  2\norm{\nabla m_{(n-1)\ell}}_{L^2(\Omega)}^2+2\lambda^2\tau^2 \ell \phi_{n-1},
\end{align*}
and substitute in Eq.~\eqref{eq:sup_discretised_m3} iteratively to find 
\begin{align*}
    \norm{\nabla m_{n\ell}}_{L^2(\Omega)}^2 &\leq 2^{n} \norm{\nabla m_0}_{L^2(\Omega)}^2+\tau^2 \lambda^2 \ell  \sum_{{k=0}}^{n-1} 2^{n-k}\phi_k.
\end{align*}
\end{proof}

Let us note that the definition of the quantity $\phi_k$ already shows that the $H^1$-control of $m$ depends on the $H^1$-control of $u$, hence the name `conditional estimate'. 
We will use it to derive a uniform $H^1$-bound for $u$ such that, a fortiori, $m \in H^1$, unconditionally.

\begin{lemma}[$L^2$-estimate for $\nabla u_\tau$]
	\label{lem:L2-gradu}
     Let $(u_k, m_k)_{k=0}^\infty \subset \A \cap (H^1(\Omega))^2$ be the solution to the implicit Euler approximation, Eq.~\eqref{eq:implicit-euler}, and let $u_\tau$ be the piecewise constant interpolation associated to $(u_k)_{k=0}^\infty$. 
     Then
     \begin{align*}
        \|\nabla u_h\|_{L^2(0,T;L^2(\Omega))}^2  \leq C,
     \end{align*}
     where $C>0$ is independent of $\tau>0$.
\end{lemma}
\begin{proof}
    Let us recall  Eq.~\eqref{eq:grad-u-by-grad-m} in Proposition~\ref{prop:discrete-entropy}:
    \begin{align*}
        \|\nabla u_k\|_{L^2(\Omega)}^2 \leq C +   \frac1\tau \prt*{\mathcal E(u_{k-1}, m_{k-1}) - \mathcal E(u_k, m_k)} + \|\nabla m_k\|_{L^2(\Omega)}^2.
    \end{align*}
    Summing from $k=n\ell+1$ to $k=(n+1)\ell$ and multiplying by $\tau>0$ yields
    \begin{align*}
        &\sum_{k=n\ell+1}^{(n+1)\ell} \tau \norm{\nabla u_{k}}_{L^2(\Omega)}^2 \\
        &\qquad\qquad\leq C \ell \tau  + \sum_{k=n\ell+1}^{(n+1)\ell} \prt*{\mathcal E(u_{k-1}, m_{k-1}) - \mathcal E(u_k, m_k)} +   \sum_{k=n\ell+1 }^{(n+1)\ell}\tau \norm{\nabla m_{k}}_{L^2(\Omega)}^2,\\
        &\qquad\qquad\leq C \ell \tau  +  \prt*{\mathcal E(u_{n\ell}, m_{n\ell}) - \mathcal E(u_{(n+1)\ell }, m_{(n+1)\ell})} +   \sum_{k=n\ell +1}^{(n+1)\ell}\tau \norm{\nabla m_{k}}_{L^2(\Omega)}^2,\\
        &\qquad\qquad\leq C (1 + \ell \tau) + \sum_{k=n\ell +1}^{(n+1)\ell}\tau \norm{\nabla m_{k}}_{L^2(\Omega)}^2, 
    \end{align*}
    using the fact that the entropy can be bounded since $0\leq u_k + m_k \leq 1$ and the domain is bounded. 
    Thus, having passed to the supremum on the right-hand side, we have
    \begin{align*}
        \sum_{k=n\ell+1}^{(n+1)\ell} \tau \norm{\nabla u_{k}}_{L^2(\Omega)}^2  
        &\leq C(1 + \ell\tau) + \ell \tau \sup_{n\ell+1 \leq k\leq (n+1)\ell}\norm{\nabla m_{k}}_{L^2(\Omega)}^2.
    \end{align*}
    Hence, recalling the definition of $\phi$ (cf. Eq.~\eqref{eq:phi}), we find 
    \begin{align*}
        \phi_n &\leq C(1+\ell \tau) +  2 \ell\tau\prt*{\norm{\nabla m_{n\ell}}_{L^2(\Omega)}^2+\tau\ell\lambda^2\phi_{n}},
    \end{align*}
    where we also used Eq.~\eqref{eq:sup-grad-m}. 
    Since both sides contain the term $\phi_n$, we rearrange and get
    \begin{align*}
        \phi_{n} &\leq \dfrac{C(1+\ell\tau) +  2\ell\tau \norm{\nabla m_{n\ell}}_{L^2(\Omega)}^2}{1 - 2\tau^2\lambda^2\ell^2}, \\
        & \leq \dfrac{C(1+\ell\tau)}{1 - 2 \tau^2 \ell^2\lambda^2}+\dfrac{2\ell\tau}{1 - 2 \tau^2\lambda^2\ell^2}\norm{\nabla m_{n\ell}}_{L^2(\Omega}^2, \\
        &\leq \dfrac{C(1+\ell\tau)}{1 - 2 \tau^2 \ell^2\lambda^2}+\dfrac{2\ell\tau}{1- 2 \tau^2\lambda^2\ell^2}\prt*{2^{n} \norm{\nabla m_0}_{L^2(\Omega)}^2+\tau^2 \lambda^2 \ell \sum_{{k=0}}^{n-1} 2^{n-k}\phi_k},
        \end{align*} 
        having used Lemma~\ref{lem:conditional-grad-m}. Thus, 
        \begin{align*}
            \phi_{n} &\leq \dfrac{C(1+\ell\tau)}{1-2\tau^2 \ell^2\lambda^2}+\dfrac{2\ell\tau}{1-2\tau^2\lambda^2\ell^2}2^n\norm{\nabla m_0}_{L^2(\Omega)}^2+\dfrac{2\tau^3\lambda^2\ell^2}{1-2\tau^2\lambda^2\ell^2}\sum_{{k=0}}^{n-1}2^{n-k}\phi_{k}.
            \label{eq:phij_normm}
    \end{align*}
    Setting $\Delta t := \ell \tau$ for $\ell>  1$, the estimate becomes
    \begin{equation*}
        \phi_{n}\leq \dfrac{C(1+\Delta t)}{1-2{\Delta t}^2\lambda^2} + \dfrac{2\Delta t}{1-2{\Delta t}^2\lambda^2}2^n\norm{\nabla m_0}_{L^2(\Omega)}^2+\dfrac{2 \tau {\Delta t}^2\lambda^2}{1-2{\Delta t}^2\lambda^2}\sum_{{k=0}}^{n-1}2^{n-k}\phi_{k}.
    \end{equation*}
    Hence, for any $\tau>0$, we choose $\ell = \ell(\tau) \in \mathbb N$ such that $\Delta t \leq 1/ (2\lambda)$ and $(\ell+1)\tau \geq 1/ (2\lambda)$. 
    This way, we obtain
    {
    \begin{align*}
        \phi_{n}& \leq C(1+\Delta t) + \tau \sum_{k=0}^{n-1}2^{n-k}\phi_{k}.\label{eq:gron_form}
    \end{align*}
    }
    Using a discrete version of Gronwall's lemma, we find that 
    \begin{equation*}
        \sup_{0\leq n\leq N}\phi_{n}\leq C e^{\sum_{k=1}^{N}2^{N-k}}\leq \bar{C_1}<\infty,
    \end{equation*}
    where $\bar C_1>0$. 
    It is crucial to highlight that $\Delta t \nrightarrow 0 $, as $\tau \to 0$. 
    Thus 
    \begin{equation*}
        \sum_{n=0}^{N}\phi_{n\ell}\leq N \bar{C_1}<\infty,
    \end{equation*}
    where $N := \ceil{T/\Delta t}$ does not go to infinity as $\tau \to 0$, thereby providing the desired uniform bound.
    As a result,
    \begin{align*}
        \|\nabla u_h\|_{L^2(0,T;L^2(\Omega))}^2 & = \int_0^T \|\nabla u_h\|_{L^2(\Omega)}^2 \dx s\\
        & \leq  \sum_{n=0}^{(N+1)l - 1} \tau \|\nabla u_n\|_{L^2(\Omega)}^2\\
        & = \sum_{n=0}^N \tau \sum_{k=nl}^{(n+1)l-1} \|\nabla u_k\|_{L^2(\Omega)}^2\\ 
        &= \sum_{n=0}^N \phi_{nl} \leq C < \infty,
    \end{align*}
    where $C>0$ does not depend on $\tau>0$.
\end{proof}

\begin{corollary}[$L^2$-estimate for $\nabla m$]
	\label{lem:L2-gradm}
    Let $(u_k, m_k)_{k=0}^\infty \subset \A \cap (H^1(\Omega))^2$ be the solution to the implicit Euler approximation (Eq.~\eqref{eq:implicit-euler}) and let $m_\tau$ be the piecewise constant interpolation associated with $(m_k)_{k=1}^\infty$. 
    Then
     \begin{align*}
        \|\nabla m_h\|_{L^2(0,T;L^2(\Omega))}^2  \leq C,
     \end{align*}
     where $C>0$ is independent of $\tau>0$.
\end{corollary}
\begin{proof}
    Revisiting Eq.~\eqref{eq:grad_mjr}, we have
\begin{align*}
    \norm{\nabla m_r}_{L^2(\Omega)}^2 &\leq 2\norm{\nabla m_0}_{L^2(\Omega)}^2 +2\tau^2 \lambda^2 r \sum_{i=1}^{r}\norm{\nabla u_i}_{L^2(\Omega)}^2\\
    &\leq C + 2\tau \lambda^2 N_T\norm{\nabla u_h}_{L^2(0,T;L^2(\Omega))}^2, 
\end{align*}
having used the initial bound on $\nabla m_0.$
Multiplying by $\tau$ and summing from $r=0$ to $N_T$, we get 
\begin{align*}
    \tau \sum_{r=0}^{N_T}\norm{\nabla m_r}_{L^2(\Omega)}^2 &\leq2 (N_T+1)\tau \norm{\nabla m_0}_{L^2(\Omega)}^2 + 2\lambda ^2 N_T (N_T+1)\tau^2 \norm{\nabla u_h}_{L^2(0,T;L^2(\Omega))}^2\\
    &\leq C,
\end{align*}
where $C>0$ is independent of $\tau.$
\end{proof}

Next, let us address the time regularity of $u$ and $m$. To this end, we introduce the notation
\begin{equation*}
    d_\tau v := \dfrac{v(x, t+\tau)-v(x,t)}{\tau}, 
\end{equation*}
for $v\in \set{u,m}$.

\begin{lemma}[Time regularity for $u_\tau$, $m_\tau$]
	\label{lem:time-reg}
    Let $(u_k, m_k)_{k=0}^\infty \subset \A \cap (H^1(\Omega))^2$ be the solution to the implicit Euler approximation (Eq.~\eqref{eq:implicit-euler}) and let $u_\tau$ be the piecewise constant interpolation associated with $(u_k)_{k=0}^\infty$. 
    Then, there holds
    $$
    	\norm{d_\tau u_\tau}_{L^2(0,T; (H^1(\Omega))')} \leq C,
    $$
    and
    $$
    	\norm{d_\tau m_\tau}_{L^2(0,T; L^2(\Omega))} \leq C,
    $$
    where $C>0$ is independent of $\tau>0$.
\end{lemma}
\begin{proof}
From Eq.~\eqref{eq:discretised_u} and employing the notation above, we see
\begin{align}
	\label{eq:discrete-u}
     &\int_\Omega d_\tau u_\tau  \varphi \dx x = -\int_\Omega \Big\{\tau {(\nabla w_{\tau} \cdot \nabla \varphi + w_{\tau} \varphi )} \nonumber \\ &\qquad\qquad\qquad\qquad\qquad+ \big[(1-u_{\tau}-m_{\tau})\nabla u_{\tau}+u_{\tau}\nabla(u_{\tau}+m_{\tau})\big] \cdot \nabla \varphi \nonumber \\ &\qquad\qquad\qquad\qquad\qquad - u_{\tau}(1-u_{\tau}-m_{\tau}) \varphi \Big\}\dx x,
\end{align}
for any $\varphi\in H^1(\Omega)$. Passing to the modulus and using Cauchy-Schwartz, we may estimate
\begin{align*}
    |\langle d_\tau u_\tau, \varphi\rangle| &\leq \tau \prt*{\norm{\nabla w_\tau}_{L^2(\Omega)} \norm{\nabla \varphi}_{L^2(\Omega)}+\norm{w}_{L^2(\Omega)}\norm{\varphi}_{L^2(\Omega)}} \\ &\qquad\qquad + \norm{\nabla \varphi}_{L^2(\Omega)}\norm{\prt*{(1-m_\tau)\nabla u_\tau+u_\tau\nabla m_\tau}}_{L^2(\Omega)} \\
    &\qquad\qquad + \norm{\varphi}_{L^2(\Omega)}\norm{u_\tau(1-u_\tau-m_\tau)}_{ L^2(\Omega)}, 
\end{align*}
whence
\begin{align*}
    |\langle d_\tau u_\tau, \varphi\rangle| 
    &\leq C_0 \|\varphi\|_{H^1(\Omega)} \prt*{\|w_\tau\|_{H^1(\Omega)} + C_1 \prt*{\norm{\nabla u_\tau}_{L^2(\Omega)}+ \norm{\nabla m_\tau}_{L^2(\Omega)}} + C_2 },
\end{align*}
having used the fact that $u_\tau, m_\tau \in L^\infty(0,T;\mathcal A\cap (H^1(\Omega))^2)$. Upon dividing by $\|\varphi\|_{H^1(\Omega)}$ and passing to the supremum, 
\begin{align*}
    \norm{d_\tau u_\tau}_{(H^1(\Omega))'} 
    &\leq C_0+  C_1\|w_\tau\|_{H^1(\Omega)} + C_2 \prt*{\norm{\nabla u_\tau}_{L^2(\Omega)}+ \norm{\nabla m_\tau}_{L^2(\Omega)}}.
\end{align*}
Squaring and integrating over $[0,T]$ yields
\begin{align*}
    &\norm{d_\tau u_\tau}_{L^2(0,T;(H^1(\Omega))')}^2 \nonumber \\  
    &\qquad\leq C \prt*{1 +  \|w_\tau\|_{L^2(0,T;H^1(\Omega))}^2 +  \prt*{\norm{\nabla u_\tau}_{L^2(0,T;L^2(\Omega))}^2+ \norm{\nabla m_\tau}_{L^2(0,T;L^2(\Omega))}^2}}\\
    &\qquad\leq C,
\end{align*}
where we used Lemma~\ref{lem:L2-gradu} and Corollary~\ref{lem:L2-gradm}. In particular, note that the constant $C> 0$ is independent of $\tau>0$. The regularity result for $d_\tau m$ follows directly from Eq.~\eqref{eq:discretised_m} and the uniform bounds on $u_\tau, m_\tau$, which completes the proof.
\end{proof}

\subsection{Existence of weak solutions}
\label{sec:exist-weak-sol}
Having established the a priori estimates, let us now show the existence of convergent subsequences whose limits we identify as weak solutions in the sense of Definition~\ref{def:weak-soln}.

The bounds provided by Lemma~\ref{lem:L2-gradu} and Corollary~\ref{lem:L2-gradm} in conjunction with the Banach-Alaoglu theorem yield the existence of subsequences and two functions $\nabla u, \nabla m \in L^2(0,T;L^2(\Omega))$, such that
\begin{itemize}
   \item $\nabla u_\tau \rightharpoonup \nabla u$ in $L^2(0, T; L^2(\Omega))$, 
   \item $\nabla m_\tau \rightharpoonup \nabla m$ in $L^2(0, T; L^2(\Omega))$,
\end{itemize}
where we did not relabel the subsequences. Moreover, by the uniform bounds of Lemma~\ref{lem:time-reg}, Lemma~\ref{lem:L2-gradu}, and Corollary~\ref{lem:L2-gradm}, we may invoke \cite[Theorem 6]{simon1986compact} such that
\begin{itemize}
    \item $u_\tau\rightarrow u$ in $L^2(0, T; L^2(\Omega))$,
    \item $m_\tau\rightarrow m$ in $L^2(0, T; L^2(\Omega))$,
\end{itemize}
again, up to subsequences. Finally, from Lemma~\ref{lem:time-reg}, we have
\begin{itemize}
    \item $d_\tau u_\tau\rightharpoonup \partial_t u$ in $L^2(0, T; (H^1(\Omega))')$,
    \item $d_\tau m_\tau\rightharpoonup \partial_t m$ in $L^2(0, T; L^2(\Omega))$,
\end{itemize}
up to a subsequence. Indeed, the limits can be identified as follows. If $d_\tau u_\tau \to \chi$, then, for $\xi(t)\phi(x) \in C_c^\infty((0,T)\times \Rd)$, we have
\begin{align*}
    \int_0^{T-\tau}\int_\Omega d_\tau u_\tau \xi(t)\phi(x) \dx x \dx t = \int_0^{T-\tau}\int_\Omega \frac{u_\tau(t+\tau) - u_\tau(t)}{\tau} \xi(t)\phi(x) \dx x \dx t
\end{align*}
Upon a change of variables, $s = t + \tau$, we find
\begin{align*}
    &\int_\Omega \phi(x) \int_0^{T-\tau}\frac{u_\tau(t+\tau) - u_\tau(t)}{\tau} \xi(t) \dx t \dx x \\
    &\qquad= \int_\Omega \phi(x) \int_\tau^T  \frac{u_\tau(x,t)}{\tau} \xi(t- \tau) \dx t \dx x - \int_\Omega \phi(x) \int_0^{T-\tau} \frac{u_\tau(x,t)}{\tau} \xi(t) \dx t \dx x\\
    &\qquad= \int_\Omega \phi(x) \int_\tau^{T-\tau} u_\tau \frac{\xi(t-\tau) - \xi(t)}{\tau} \dx t \dx x + \frac1\tau \int_\Omega \phi(x) \int_{T-\tau}^T u_\tau(x,t) \xi(t-\tau) \dx t \dx x \\ 
    &\qquad\qquad - \frac1\tau\int_\Omega \phi(x) \int_0^\tau u_\tau(x,t) \xi(t) \dx t \dx x\\
    &\qquad\longrightarrow -\int_\Omega \phi(x) \int_0^T u(x,t) \xi'(t) \dx t \dx x,
\end{align*}
\textit{i.e.}, $\chi = \partial_t u$ in the sense of distributions. Replacing $u$ by $m$ and setting $d_\tau m_\tau \rightharpoonup \tilde \chi$, following the above argument, we have
\begin{align*}
	\int_0^{T} \tilde \chi(t)  \xi(t) \dx t  = - \int_0^{T} m(t) \xi'(t) \dx t,
\end{align*}
and thus $m' = \tilde \chi \in L^2(0,T; L^{2}(\Omega))$.

Having garnered sufficient compactness and the corresponding convergent subsequences and limits, we can now prove the main result.
\begin{proof}[Proof of Theorem~\ref{thm:main}]
Let us revisit Eq.~\eqref{eq:discrete-u}, \textit{i.e.},
\begin{align*}
    \int_{0}^{T} \int_\Omega d_\tau u_\tau \varphi\, \dx x \dx t &= -\tau \int_{0}^{T} \int_{\Omega}\nabla w \cdot \nabla \varphi + w \varphi \, \dx x \dx t \\
    &\qquad\qquad - \int_{0}^{T} \int_{\Omega} \prt*{(1-m_\tau)\nabla u_\tau +u_\tau \nabla m_\tau} \cdot \nabla \varphi \, \dx x \dx t \\
    &\qquad\qquad + \int_{0}^{T} \int_{\Omega}  u_\tau (1-u_\tau -m_\tau) \varphi\, \dx x \dx t.
\end{align*}
First let us note that the term premultiplied by $\tau$ vanishes due to Corollary~\ref{cor:L2H1-estimate-w}. Next, using the convergences above, we can pass to the limit in the other terms of the equation to get
\begin{align*}
    \int_{0}^{T}\int_\Omega \partial_t u \varphi\, \dx x \dx t 
    &= - \int_{0}^{T}\int_{\Omega} \prt*{(1-m)\nabla u +u \nabla m} \cdot \nabla \varphi \, \dx x \dx t\\
    &\qquad\qquad +\int_{0}^{T}\int_{\Omega} u (1-u -m) \varphi \, \dx x \dx t,
\end{align*}
for any $\varphi \in C_c^\infty(\R^d \times (0,T))$ which is dense in $L^2(0,T;H^1(\Omega))$. Thus, the limit $(u,m)$ is a weak solution to Eq.~\eqref{eq:main-cells} in the sense of Definition~\ref{def:weak-soln}. Similarly, we can pass to the limit in the equation for the ECM, Eq.~\eqref{eq:main-ECM}, in duality with $L^2(0,T;L^2(\Omega))$, \textit{i.e.}, we get
\begin{align*}
    \int_{0}^{T} \int_\Omega \partial_t m \, \varphi  \, \dx x \dx t &= -\lambda \int_0^T \int_\Omega u m \varphi \, \dx x \dx t,
\end{align*}
for any $\varphi \in L^2(0,T;L^2(\Omega))$, having used the same approximation procedure of the test function.

The a priori estimates \eqref{eq:a_priori} follow from passing to the limit in the bounds of Lemmas \ref{lem:L2-gradu} and \ref{lem:L2-gradm}, using the weak lower semicontinuity of the norms.\\
Finally, the compactness is sufficient to conclude that the weak solution satisfy the initial data, using in addition that $m_0 \to m_{\mathrm{in}}$ pointwise as $\tau \to 0$.
\end{proof}

\section{A numerical exploration of large ECM degradation rates}\label{sec:numerics}

In this section, we explore solutions to the system~\eqref{eq:main} numerically, using simulations subject to no-flux boundary conditions and the following initial conditions:
\begin{equation}
    u({\bf{x}},0)=\begin{cases}1, \qquad &$if$ \qquad |{\bf{x}}|<1, \\ 0, \qquad &$if$ \qquad |{\bf{x}}|\geq1,\end{cases}\label{IC_u} 
\end{equation}
\begin{align}
    m({\bf{x}},0)= \begin{cases} 
    0, \qquad &$if$ \qquad |{\bf{x}}|< 1, \\
    m_0, \qquad &$if$ \qquad |{\bf{x}}|\geq 1. \label{IC_m}
    \end{cases} 
\end{align}
    
As in \cite{crossley2023travelling}, the system~\eqref{eq:main} is solved in one dimension on the domain $x\in[0, L]$, where $L\geq200$ is chosen sufficiently largely to ensure a full travelling wave type profile had evolved without the impact of boundary conditions. 
In two dimensions, we solve the system~\eqref{eq:main} numerically on the domain ${\bf{x}}=(x, y)$ where $x,y\in [-5, 5]$.
The spatial discretisation used is $\Delta x = 0.1$, where we employ the method of lines to discretise physical space before integrating in time using Python's built-in integrator \texttt{scipy.intergrate.solve\_ivp}, which uses an explicit Runga-Kutta integration method of order 5 and time step $\Delta t=1.$
{This specification of $\Delta t, \, \Delta x$ satisfies the CFL condition for this system of equations~\eqref{eq:main}, given by
\begin{equation*}
    \Delta t \leq \dfrac{(\Delta x)^2}{2 \max\left(\lvert 1 - u - m \rvert, \lvert u \rvert\right)}.
\end{equation*}}
The discretisation employed at spatial point $i$ is
\begin{equation*}
    \frac{\partial}{\partial x}\bigg[D\frac{\partial a}{\partial x}\bigg]\approx \frac{1}{2(\Delta x)^2}\bigg[(D_{i-1}+D_{i})a_{i-1}-(D_{i-1}+2D_{i}+D_{i+1})a_i+(D_i+D_{i+1})a_{i+1}\bigg].
\end{equation*}

When simulating the system of equations~\eqref{eq:main} on a fixed domain, subject to zero flux boundary conditions and initial conditions~\eqref{IC_u}~and~\eqref{IC_m}, we observe solutions in both one and two spatial dimensions that we have a travelling wave type profile, whose far-field conditions are determined by the steady states of the system of equations~\eqref{eq:main}: 
\begin{equation*}
    (u,\, m) = (1,\, 0), \, (0,\, m_0).
\end{equation*}
In the travelling wave type profile, the cell density, $u$, decreases monotonically from one to zero, as the cells diffuse into available space and proliferate up to the carrying capacity behind the wave, driving invasion. 
The ECM density, $m$, increases monotonically between zero (where it is completely degraded by cells in the same spatial position as the ECM) and $m_0$ (where no degradation has occurred as no cells are present here yet, ahead of the wave).

One key feature of the profiles in cell and ECM densities is the overlap between the wave fronts. 
For low ECM degradation rates, we observe a non-negligible region of overlap in the cell density and ECM profiles. 
Alternatively, as $\lambda$ increases, the very low density of cells at the front of the non-compactly supported cell density profile (in the exponential tail) have a larger impact on the ECM. 
This is because they are able to degrade the ECM locally at a much higher rate and therefore the remaining density of ECM is significantly reduced, resulting in a much smaller overlap in the travelling wave type profiles of cell and ECM densities. 
Indeed, in extremely high ECM degradation rates, a gap appears between the cell and ECM density wave fronts, typical of those observed in acid-mediated tumour invasion models, such as those by Gatenby and Gawlinski and others \cite{gatenby1996reaction, gatenby_acid-mediated_2006, martin2010tumour}. 
It is also clear from Figures~\ref{fig:2_1d},~\ref{fig:2d},~\ref{fig:gaussian2d}~and~\ref{fig:sin1d} that increasing the ECM degradation rate impacts the shape of the ECM density travelling wave type profile. 
As $\lambda$ increases, the ECM profile displays a sharper transition between zero and $m_0$ around the wave front. 

\begin{figure}[htbp]
    \centering
    \includegraphics[width=\linewidth]{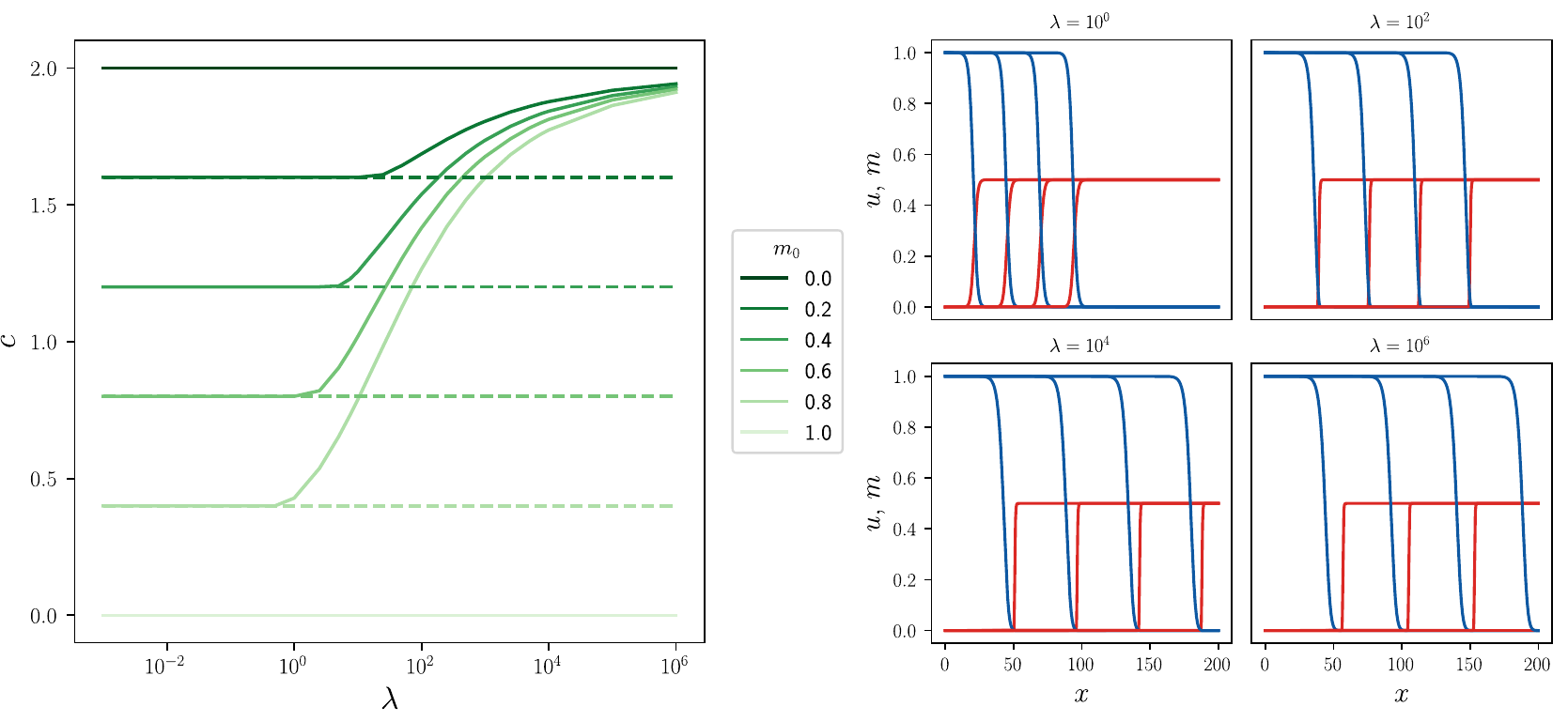}
    \caption{On the left, plot of the relationship between the {ECM degradation rate, $\lambda$, and the} numerically estimated travelling wave speed{, $c$,} (solid lines) of solutions of the system \eqref{eq:main} subject to zero flux boundary conditions and initial conditions~\eqref{IC_u}~and~\eqref{IC_m} {for various initial ECM densities, $m_0$}. The numerically estimated travelling wave speed is obtained by interpolation to find the point $X(t)$ such that $u(X(t), t) = 0.1$ at all times $t\geq0$. The dashed lines indicate the value of the analytically predicted minimum travelling wave speed, $2(1-m_0)$ \cite{crossley2023travelling}.
    On the right, plots of solutions to system \eqref{eq:main} subject to zero flux boundary conditions and initial conditions~\eqref{IC_u}~and~\eqref{IC_m} in one dimension for the cells (blue) and  for the ECM (red), where $m=0.5$. 
    Plots are shown subject to ECM degradation rates $\lambda=10^0,\, 10^2, \, 10^4, \, 10^6$ at 
    times $t=25, \, 50, \, 75\, \text{and}\, 100,$ from left to right.}
    \label{fig:2_1d}
\end{figure}

In Figures~\ref{fig:2_1d} and \ref{fig:2d}, it can be observed that for a constant initial condition for the ECM density, the solutions for the cell and ECM density have a constant speed and constant profile. 
Initially, as the cell density transitions from the initial condition to the travelling wave type profile, we observe a speed that increases with time. 
However, once the cell density reaches its constant, travelling wave type profile, it propagates through the spatial domain with a constant speed. 
At low values of the ECM degradation rate, $\lambda$, the numerically observed travelling wave speed matches that predicted by standard travelling wave analysis, via linearisation of the system of ODEs associated with Eqs.~\eqref{eq:main-cells}~and~\eqref{eq:main-ECM}.
This analytically predicted minimum travelling wave speed, $c_{\text{analytical}}=2\sqrt{(1-m_0)}$, is derived in \cite{crossley2023travelling}, and, importantly, {although describes an explicit dependence on the initial ECM density, it} is independent of the ECM degradation rate, $\lambda$. 
However, the plot on the left in Figure~\ref{fig:2_1d} demonstrates that the speed of invasion of the cells does, in reality, depend on $\lambda$, as supported by the plots on the right in Figure~\ref{fig:2_1d} and Figure~\ref{fig:2d}, which show solutions of Eqs.~\eqref{eq:main-cells}~and~\eqref{eq:main-ECM} for different values of the ECM degradation rate, $\lambda$. 
From this, it is clear that increasing $\lambda$ increases the speed of invasion of the cells, such that the numerically estimated minimum travelling wave speed, $c_{\text{numerical}}\to2^{-}$ as $\lambda\to\infty$.

\begin{figure}[htbp]
    \centering
    \includegraphics[width=\linewidth]{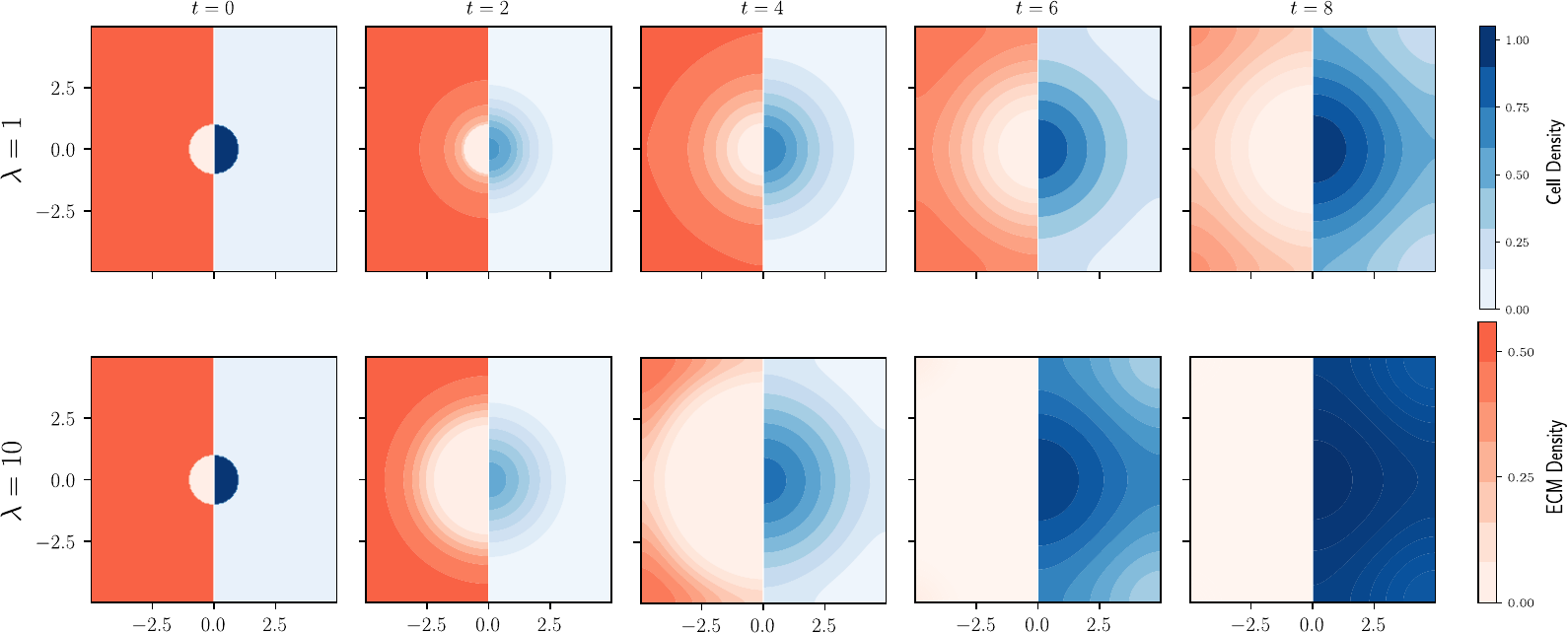}
    \caption{Plots of solutions to system \eqref{eq:main} subject to zero flux boundary conditions and initial conditions~\eqref{IC_u}~and~\eqref{IC_m} for the cells (blue, {plotted only} on the right half of the plots) and  for the ECM (red, {plotted only} on the left half of the plots), where $m_0=0.5$. {Results are symmetric, hence plotted in this way.} Plots are shown subject to ECM degradation rates $\lambda=1,\, 10$ at times $t=0,\,2,\,4,\,6,\,8.$}
    \label{fig:2d}
\end{figure}

The {varying} discrepancy between the analytically predicted minimum travelling wave speed $c_{\text{analytical}}=2\sqrt{(1-m_0)}$ and the numerically observed minimum travelling wave speed, $c_{\text{numerical}}$, {for increasing $\lambda$ (which can be observed on the left in Figure~\ref{fig:2_1d})} is primarily due to the linearisation of the system during travelling wave analysis, which disregards the complexity of the non-linear diffusive terms {and ignores the impact of parameters, such as $\lambda$, which are only present in the coupled equation describing ECM evolution over time}. 
Results like these, where multi-species systems with cross-dependent diffusion show a discrepancy between the analytically predicted and numerically estimated travelling wave speeds are also observed in coupled systems of PDEs such as those {studied} in \cite{crossley2024phenotypic, stepien2018traveling} and is discussed in more detail in \cite{falco2024travelling}. 

\begin{figure}[htbp]
    \centering
    \includegraphics[width=\linewidth]{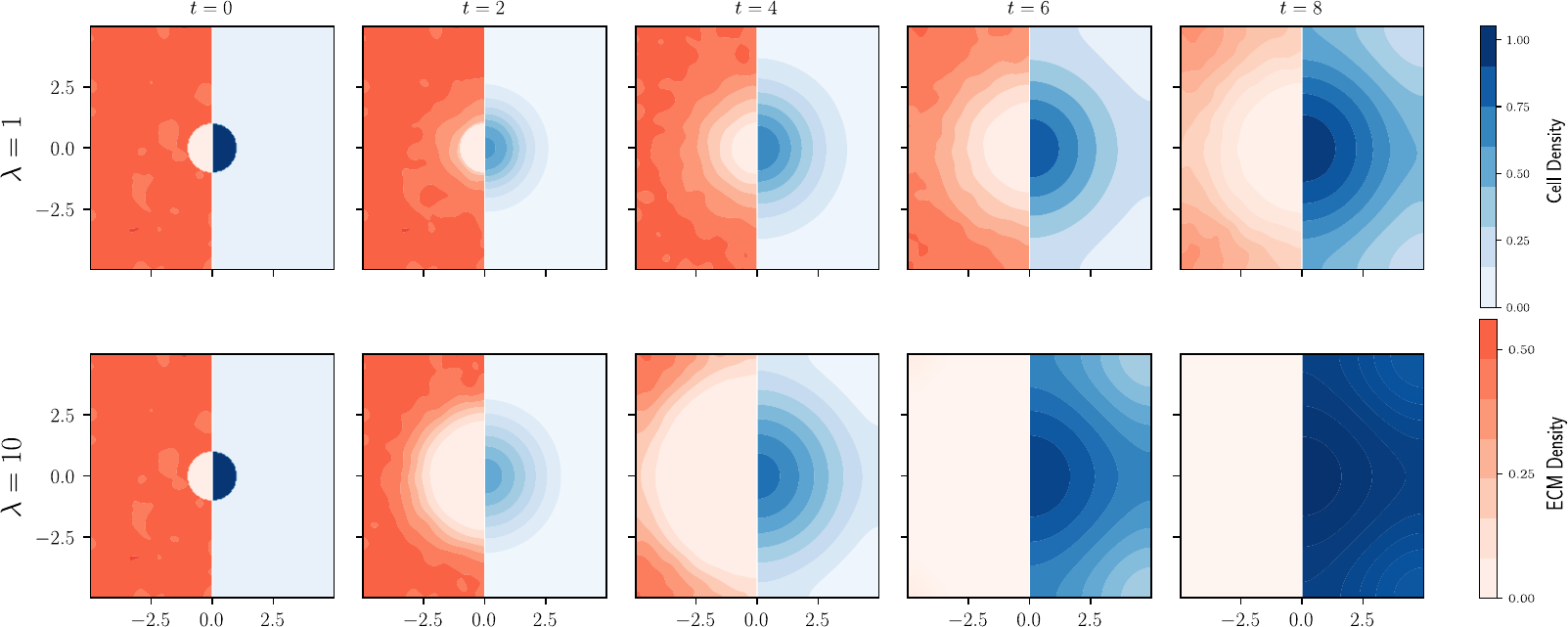}
    \caption{Plots of solutions to system \eqref{eq:main} subject to zero flux boundary conditions and initial conditions for the cells (blue, {plotted only} on the right half of the plots) as in Eq.~\eqref{IC_u} and random initial data for the ECM, smoothed using a Gaussian filter with $\sigma=5$ (red, {plotted only} on the left half of the plots). {Results are symmetric, hence plotted in this way.} Plots are shown subject to ECM degradation rates $\lambda=1,\, 10$ at times $t=0,\,2,\,4,\,6,\,8.$}
    \label{fig:gaussian2d}
\end{figure}

If we instead consider a setup where the initial condition for the ECM is inhomogeneous across the space ahead of the cells, we observe a less smooth profile in the ECM density. 
Consider Figure~\ref{fig:gaussian2d}, where we simulate the system~\eqref{eq:main} with a random initial condition for the ECM, which is smoothed under a Gaussian filter (using \texttt{scipy.ndimage.gaussian\_filter} with $\sigma=5$). 
By examining the ECM (plotted on the left, in red), we can see that small, stochastic differences in the initial ECM density develop into larger regions of disparity in the ECM profile as it is degraded by invading cells, such that the degraded front in the ECM density varies azimuthally during a two dimensional simulation.  
In contrast to this, the radial migration of the cell density ({plotted on the right, in }blue) smooths out the azimuthal variation in the ECM density profile and thus a smooth invading cell population is observed. 
This is likely due to the higher cell densities, and a greater impact can be observed when the ECM degradation rate is larger.
In fact, the travelling wave type profile observed for the cells when the surrounding ECM density is constant (see Figure~\ref{fig:2d}) or random (see Figure~\ref{fig:gaussian2d}) are almost identical for a given value of the ECM degradation rate shown in Figure~\ref{fig:gaussian2d}. 

To investigate whether this averaging out among the invading cell density is always observed, we returned to one dimension and considered an oscillatory initial condition for the ECM. 
When oscillations are small in amplitude or high in frequency, similar behaviour to the random ECM setup in two dimensions was observed - \textit{i.e.}, the cells invaded with a constant speed, constant profile solution, and quickly degraded the ECM. 
In contrast, when we simulated an initial condition for the ECM with a larger amplitude and lower frequency of oscillations, we notice that there were constant profile solutions, however the speed of invasion varied through time and space (see Figure~\ref{fig:sin1d}). 
As the ECM degradation rate increases, {the} differences in the speed of the travelling wave type profiles {decreases, and only marginal changes can be observed over time}.  
However, for small ECM degradation rates {(as shown on the left in Figure~\ref{fig:sin1d})}, we observe a large disparity in the speed of invasion throughout the spatial domain{. A}s the cells move through regions of low ECM density, the speed of invasion increases, and then, in regions of higher ECM density, the speed of invasion decreases again. {The observed speeds in each region match those shown on the left in Figure~\ref{fig:2_1d}.}

\begin{figure}[htbp]
    \centering
    \includegraphics[width=\linewidth]{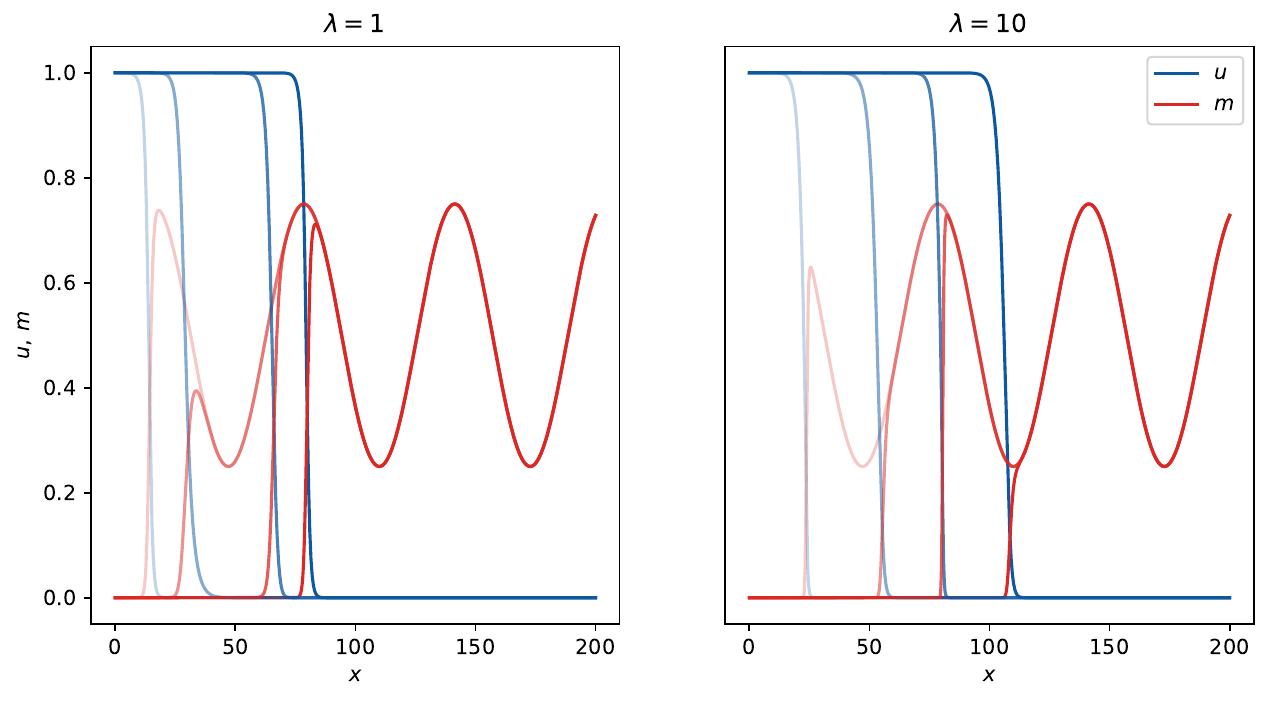}
    \caption{Plots of solutions to system \eqref{eq:main} subject to zero flux boundary conditions and initial conditions for the cells given by Eq.~\eqref{IC_u}, but for the ECM as $m(x,0)=0.5+0.25\,\text{sin}(x/10)$. Solutions are shown in one dimension for the cells (blue) and  for the ECM (red) at 
    times $t=25, \, 50, \, 75\, \text{and}\, 100,$ from left to right. }
    \label{fig:sin1d}
\end{figure}

In future work, the discrepancy between the analytically predicted and numerically estimated minimum travelling wave speed could be investigated, with the aim to uncover the relationship between $m_0$, the ECM density ahead of the invading cells, and $\lambda_c$, the critical value of the ECM degradation rate where $c_{\text{numerical}}>c_{\text{analytical}}$ when $\lambda > \lambda_c$.
It would also be of interest to formally show that solutions to Eqs.~\eqref{eq:main} converge to solutions of the Fisher-KPP equation, with $c=2$, when $\lambda\to\infty.$

\subsection*{Acknowledgements}
RMC is supported by funding from the Engineering and Physical Sciences Research Council (EP/T517811/1) and the Oxford-Wolfson-Marriott scholarship at Wolfson College, University of Oxford.
RMC would also like to thank the Institute of Scientific Computing at the Technische Universit\"at Dresden for their funding and kind hospitality during her collaborative visit, where work on this paper was undertaken.
JFP thanks the DFG for support via the Research Unit FOR 5387 POPULAR, Project No. 461909888.

\bibliographystyle{plain}
\bibliography{arxiv-update}
\end{document}